%% file: main_arXiv.tex
\documentclass[11pt,english]{article}
\usepackage{lmodern}

\usepackage[T1]{fontenc}
\usepackage[latin9]{inputenc}
\usepackage{geometry}
\usepackage{fancybox}
\usepackage{calc}
\usepackage{units}
\usepackage{dsfont}
\usepackage{amsmath}
\usepackage{amsthm}
\usepackage{amssymb}

\makeatletter
\numberwithin{figure}{section}
\numberwithin{equation}{section}
\theoremstyle{plain}

\theoremstyle{plain}

\usepackage{tikz}
\usetikzlibrary{arrows.meta,positioning,shapes.geometric,calc,fit}

\usepackage{comment,url,algorithm,algorithmic,graphicx,subcaption,relsize}
\usepackage{amssymb,amsfonts,amsmath,amsthm,amscd,dsfont,mathrsfs,mathtools,microtype,nicefrac,pifont}
\usepackage{float,psfrag,epsfig,color,url,hyperref}
\hypersetup{
  colorlinks,
  linkcolor={red!50!black},
  citecolor={blue!50!black},
  urlcolor={blue!80!black}
}
\usepackage{upgreek}
\usepackage{epstopdf,bbm,mathtools,enumitem}
\usepackage[toc,page]{appendix}
\usepackage{dsfont}

\usepackage[mathscr]{euscript}
\usepackage{natbib}
\bibpunct{(}{)}{;}{a}{,}{,}
\usepackage{custom}

\makeatletter
\renewcommand{\paragraph}{%
  \@startsection{paragraph}{4}%
  {\z@}{1.25ex \@plus 1ex \@minus .2ex}{-1em}%
  {\normalfont\normalsize\bfseries}%
}
\makeatother

\makeatother

\usepackage{babel}
\providecommand{\lemmaname}{Lemma}
\providecommand{\theoremname}{Theorem}

\begin{document}
\input{math-macros.tex}

\title{Sharp propagation of chaos in R\'enyi divergence}

 \author{
Matthew S.\ Zhang\thanks{
  Department of Computer Science at
  University of Toronto, and Vector Institute, \texttt{matthew.zhang@mail.utoronto.ca}
}
}

\maketitle
\begin{abstract}
We establish sharp rates for propagation of chaos in \emph{R\'enyi divergences} for interacting diffusion systems at stationarity. Building upon the entropic hierarchy established in~\cite{lacker2023hierarchies}, we show that under strong isoperimetry and weak interaction conditions, one can achieve $\Renyi_q(\mu^1 \mmid \pi) = \widetilde O(\frac{d q^2}{N^2})$ bounds on the $q$-R\'enyi divergence.
\end{abstract}

\section{Introduction}

In this work, we study propagation of chaos in divergences stronger than entropy. Consider minimization of energy functionals of the following type,
\begin{align}\label{eq:composite-energy}
    \mc E(\mu) = \int V \, \D \mu + \frac{1}{2} \iint W(x-y) \, \D \mu^{\otimes 2}(x, y) + \int \log \mu \, \D \mu\,.
\end{align}
Write $\pi$ for the minimizer of this energy functional, which exists and is unique under some convexity conditions of $\mc E$ over the space of measures. If such a minimizer does exist, then it satisfies the following implicit equation:
\begin{align}\label{eq:mf-minimizer}
    \pi \propto \exp\Bigl(-V - \int W(\cdot - z) \, \D \pi(z)\Bigr)\,.
\end{align}

An alternate viewpoint describes $\pi$ as the stationary law of a particle following the McKean--Vlasov equation, which describes the Wasserstein gradient flow of $\mc E$,
\begin{align}\label{eq:mckean-vlasov}
    \D Y_t = \Bigl( - \nabla V(Y_t) - \int \nabla W(Y_t - \cdot) \, \D \operatorname{law}(Y_t)(\cdot)\Bigr) \, \D t + \sqrt{2} \, \D B_t\,.
\end{align}
The presence of an integral against $\operatorname{law}(Y_t)$ in~\eqref{eq:mckean-vlasov} makes it difficult to simulate. To correct this, one considers a system of evolving particles, $(X_t^i)_{t \geq 0, i \in [N]}$, with each $X_t^i \in \R^d$, satisfying for independent Brownian motions $(B_t^i)_{i \in [N], t \geq 0}$,
\begin{align}\label{eq:finite-particle}
    \D X_t^i = \Bigl( - \nabla V(X_t^i) - \frac{1}{N-1} \sum_{\substack{j \in [N] \\ j \neq i}} \nabla W(X_t^i- X_t^j)\Bigr) \, \D t + \sqrt{2} \, \D B_t^i\,.
\end{align}
Again, under mild conditions, this system admits a stationary measure
\begin{align}\label{eq:finite-particle-stat}
    \mu^{1:N}(X^{1:N}) \propto \exp\Bigl(-\sum_{i=1}^N V(X^i) -\frac{1}{2(N-1)} \sum_{i \in [N], j \in [N] \backslash i} W(X^i - X^j) \Bigr)\,.
\end{align}
One is naturally led to inquire how close $\mu^{1:N}$ is to $\pi^{\otimes N}$ under appropriate assumptions on $V, W$. This general question is known as (stationary) \textbf{propagation of chaos}\footnote{More canonically, propagation of chaos is typically used to show evolving bounds along the dynamic systems given by~\eqref{eq:mckean-vlasov} and~\eqref{eq:finite-particle}, which can be used to demonstrate properties at stationarity. However, the two phenomena are related and the techniques used are not particularly different.}.

Early results by~\cite{sznitman2006topics, malrieu2003convergence} have led to quantitative bounds on the quality of this approximation, both in Wasserstein and relative entropy senses, which follow from a subadditivity type result $\sum_{i=1}^n W_2^2(\mu^i, \pi) \leq W_2^2(\mu^{1:N}, \pi^{\otimes N})$. However, the result of~\cite{lacker2023hierarchies} and subsequent work~\citep{lacker2023sharp, ren2024size} has established that, under certain convexity conditions, we have bounds $\KL(\mu^1 \mmid \pi) \lesssim 1/N^2$, which is much smaller than the $1/N$ typically obtained via central limit theorems or subadditivity.

\paragraph{R\'enyi divergences.}
In this work, we are interested in stronger notions of convergence: namely, the $q$-R\'enyi divergences. For two measures $\mu, \nu$ with $\mu \ll \nu$, we define for $q > 1$
\begin{align}\label{eq:renyi}
    \Renyi_q(\mu \mmid \nu) \deq \frac{1}{q-1} \log \E_\nu \bigl(\frac{\D \mu}{\D \nu}\bigr)^q\,,
\end{align}
with $\lim_{q \searrow 1} \Renyi_q(\mu \mmid \nu) = \KL(\mu \mmid \nu)$.
These divergences are a natural strengthening of the relative entropy, with guarantees in $q$-R\'enyi immediately implying closeness in relative entropy and total variation distance (and the Wasserstein distance under a transport inequality). Furthermore, they allow for fine-grained control of tail probabilities, which has implications for statistical procedures which use~\eqref{eq:finite-particle-stat} as an approximation for~\eqref{eq:mf-minimizer} (or vice versa). See Corollary~\ref{cor:change-measure} and the following remark for some further discussions.

Establishing propagation of chaos in R\'enyi divergence is not simple. Firstly, the classical argument for both relative entropy and Wasserstein distances bounds the relative entropy $\KL(\mu^{1:N} \mmid \pi^{\otimes N}) \lesssim O(1)$, and then uses subadditivity to conclude. This, however, does not work in R\'enyi divergence, as the analogous subadditivity property does not hold in general, even when the second argument is a product measure.

The recent work of~\cite{hess2025higher} studies propagation of chaos in $\chi^2$-divergence for interactions in the dynamical setting (i.e., for the solutions evolving along~\eqref{eq:finite-particle} and~\eqref{eq:mckean-vlasov} outside of stationarity), which they then apply towards deriving measures which better approximate (compared to the mean-field dynamics) the finite-particle dynamics. This also relies on a hierarchical analysis, but one which is highly dependent on the boundedness of the interaction kernel. This is not easily extended to more general cases, and does not seem to generalize well beyond $q = 2$.

Instead, we aim to directly apply the hierarchical approach of~\cite{lacker2023hierarchies}. However, compared to the hierarchy in~\cite{lacker2023hierarchies}, we cannot apply transport and log-Sobolev inequalities immediately to obtain a recursive system. This is because the R\'enyi analogue of the Fisher information involves the expectation of the square norm of the score $\nabla \log \frac{\mu^1}{\pi}$ under a reweighted measure, which does not easily compare to a Wasserstein distance or another R\'enyi divergence.

Through some analytic innovations, we show that it suffices to demonstrate that the relative score $\nabla \log \frac{\mu^{1}}{\pi}$ is sub-Gaussian under $\mu^1$ with variance proxy which diminishes as $1/N^2$, for $N \gg 1$. To study this, we provide bounds on the Lipschitz constant of the relative score. This relies on a hierarchical argument which has a similar flavour as previous work~\citep{lacker2023hierarchies}, but which aims to bound the Wasserstein distance between two conditional measures $\mu^{k+1 \mid [k]}(\cdot \mid x^{[k]})$, $\mu^{k+1 \mid [k]}(\cdot \mid \bar x^{[k]})$ in terms of the distance of their conditioning points, $\norm{x^{[k]} - \bar x^{[k]}}^2$.

This allows us to successfully study the propagation of chaos in R\'enyi divergences, showing the following:
\begin{center}
If $\mu$, $\pi$ satisfy sufficiently strong isoperimetric and smoothness conditions, then $\Renyi_q(\mu^1 \mmid \pi) \lesssim \widetilde O(\frac{d q^2}{N^2})$, for $N$ sufficiently large.
\end{center}
This answers a conjecture posed in~\citet[Section 5]{kook2024sampling}, and also allows one to quantify the complexity of obtaining samples close to $\pi$ in $q$-R\'enyi divergence. Furthermore, it serves as a natural extension of the work in~\cite{lacker2023sharp}, translating to bounds on $\norm{\frac{\D \mu^1}{\D \pi}}_{L^q(\pi)}$ for $q \geq 1$.

\subsection{Related work}

The study of \emph{propagation of chaos} has a venerable history which is not easily summarized; we touch only upon the most relevant works below. Beginning with the seminal work of~\cite{sznitman2006topics} and~\cite{malrieu2001logarithmic, malrieu2003convergence}, quantitative chaos bounds on the order of $1/N$ in $W_2^2$ or $\KL$ divergence have been shown by first bounding the divergence between $\mu^{1:N}$ and $\pi^{\otimes N}$ by something that is independent of $N$, and then applying subadditivity. This line of thought has been successfully generalized to the case of singular interactions of $1/\norm{x^i - x^j}^\alpha$ type~\citep{jabin2017mean, jabin2018quantitative, bresch2023mean}, $\alpha > 0$, which permits the study of Coulomb and Riesz gases, as well as other physical systems. See~\cite{chaintron2022propagation} for a more complete survey.

A surprising result from~\cite{lacker2023hierarchies} showed that this rate could be improved to $\widetilde{O}(k^2/N^2)$ under a ``weak interaction'' assumption (requiring a trade-off between the isoperimetric constant of $\pi$, and the smoothness of $W$) which is sharp for e.g., quadratic interactions and potentials. This was further developed in~\cite{lacker2023sharp, grass2025sharp} to handle more general settings, and in~\cite{grass2025propagation} for the Fisher information. The work of~\cite{ren2024size} also significantly expanded the scope, removing the ``weak interaction'' assumption by relying on a ``non-linear'' log-Sobolev inequality.

We also mention that a similar ``propagation of chaos'' for minimizers of functionals of a more general form has been proposed and studied in works such as~\cite{chen2025uniform, chen2024uniform, wang2024uniform}. This covers some applications such as training dynamics for two-layer neural networks which do not fall under~\eqref{eq:composite-energy}. However, this setting has not proven amenable to analysis using hierarchical techniques, and the $1/N^2$ rate remains elusive in this setting.

The most relevant work to ours is~\cite{hess2025higher}, which also provides a means of establishing higher order propagation of chaos in the dynamical setting. When the interaction is a.s.\ bounded, they show the sharp rate of $1/N^2$ via hierarchical techniques reminiscent of~\cite{lacker2023hierarchies} in the $\chi^2$-divergence (equivalent to $2$-R\'enyi), and which do not contain a ``weak interaction'' assumption. However, their analysis cannot be easily adapted to the case of unbounded-but-smooth interactions considered in this work. Finally,~\cite{bresch2022new} offer a different proof for dynamical propagation of chaos in some $\chi^p$-norms, but which only works for short times and without quantitative $N$-dependence.

\section{Preliminaries}

\subsection{Notation}
We use the notation $a \lesssim b$, $a = O(b)$ if $a \leq C b$ for a constant $C$ for an absolute constant $C$ (i.e., a constant which is independent of $d, N, k, q, \bar C_{\LSI}, \beta_W$, etc.). Similarly, we use $a \gtrsim b, a = \Omega(b)$, if $a \geq Cb$, and $a \asymp b, a = \Theta(b)$ if $a \lesssim b$ and $a \gtrsim b$ simultaneously. Finally, we use $a = \Otilde(b)$ if $a \lesssim b \log^C b$ where $C$ is again an absolute constant, and define $\widetilde \Theta$, $\widetilde \Omega$ analogously. We use $[k]$ to denote the set $\{1, 2, \dotsc, k\}$. We use the following vector notation, with $X^{1:N} = [X^1, \dotsc, X^N] \in \R^{d \times N}$, $X^{j:k} = [X^j, X^{j+1}, \dotsc, X^k]$ for $j \leq k$, $X^{-i} = [X^1, \dotsc, X^{i-1}, X^{i+1}, \dotsc, X^N]$. Similarly, we use this notation for the measures associated with these random variables, e.g.\  $\mu^{i \mid -i}(\cdot \mid x^{-i})$ for the measure of $X^i$ given $X^{-i} = x^{-i}$, $x^{-i} \in \R^{d \times (N-1)}$.

\subsection{Background}

\paragraph{Divergences/distances between probability measures}

For $\mu, \nu \in \mc P(\R^d)$ with finite second moments, we define
\begin{align*}
    W_2^2(\mu, \nu) = \inf_{\gamma \in \Gamma(\mu, \nu)} \E_{(X, Y) \sim \gamma}[\norm{X-Y}^2]\,,
\end{align*}
where $\Gamma(\mu, \nu)$ is the set of measures on the product space $\R^d \times \R^d$ with marginals $\mu, \nu$ respectively. We define the relative entropy and Fisher information for $\mu \ll \nu$ as, respectively
\begin{align*}
    \KL(\mu \mmid \nu) = \E_{\mu} \log \frac{\D \mu}{\D \nu}\,, \qquad \FI(\mu \mmid \nu) = \E_\mu \Bigl[\Bigl\lVert\,\nabla \log \frac{\D \mu}{\D \nu}\,\Bigr\rVert^2\Bigr]\,.
\end{align*}
Recalling the definition of the R\'enyi divergence in~\eqref{eq:renyi}, we define the R\'enyi Fisher information as, with $\rho \deq \mu/\nu$ as usual,
\begin{align*}
    \RFI_q(\mu \mmid \nu) \deq \frac{4}{q} \frac{\E_\nu[\norm{\nabla (\rho^{q/2})}^2]}{\E_\nu[\rho^q]} = q \E_\mu[\psi^q \norm{\nabla \log \rho}^2]\,, \qquad \psi^q \deq \frac{\rho^{q-1}}{\E_\nu[\rho^q]}\,.
\end{align*}

\paragraph{Isoperimetric inequalities}
We say a measure $\pi$ on $\R^d$ satisfies a \emph{log-Sobolev} inequality which holds with constant $C_{\LSI}$ if for all compactly supported and smooth $f: \R^d \to \R$,
\begin{align}\label{eq:lsi}\tag{$\msf{LSI}$}
    \operatorname{ent}_\pi f \deq \E_\pi \Bigl[f \log \frac{f}{\E_\pi f} \Bigr] \leq \frac{C_{\LSI}}{2}\, \E_\pi[\norm{\nabla f}^2]\,.
\end{align}
Additionally, we say that a measure $\pi \in \mc P_{2, \operatorname{ac}}(\R^d)$ satisfies Talagrand's inequality with constant $C_{\msf T_2}$, which is implied by~\eqref{eq:lsi} with the same constant, if for all measures $\mu \in \mc P_{2, \operatorname{ac}}(\R^d)$ we have
\begin{align}\label{eq:talagrand}\tag{$\msf T_2$}
    W_2^2(\mu, \pi) \leq 2C_{\msf T_2} \KL(\mu \mmid \pi)\,.
\end{align}

\subsection{Main results}

We require a smoothness assumption on the potential $W$.
\begin{assumption}\label{as:smoothness}
    The potential $W$ is $\beta_W$-smooth, implying that for all $x, y \in \R^d$,
    \begin{align*} 
        \norm{\nabla W(x) - \nabla W(y)} \leq \beta_W \norm{x-y}\,.
    \end{align*}
\end{assumption}
We also assume a notion of strong isoperimetry, given below.
\begin{assumption}\label{as:isoperimetry}
     Assume that $\mu^{k+1:k+\ell \mid [k]}(\cdot \mid x^{[k]})$ satisfies~\eqref{eq:lsi} with constant $\bar C_{\LSI}$ uniformly in $x^{[k]} \in \R^{d \times k}$ for all $k \in [N-1]$, $\ell \in [N-k]$, and that a minimizer $\pi$ exists for~\eqref{eq:composite-energy} and satisfies~\eqref{eq:lsi} with the same constant. Furthermore, suppose that $\mu^{[k]}$ also satisfies~\eqref{eq:lsi} with the same constant.
\end{assumption}
\begin{remark}
    This requires that the finite-particle measure $\mu^{[k]}$ satisfy~\eqref{eq:lsi} with a constant independent of $N$, for which it suffices that $\mu^{[N]}$ satisfy~\eqref{eq:lsi} with the same constant. See \S\ref{scn:examples} for some discussion of when this holds. 
\end{remark}

We now introduce the following notion of very weak interaction, which is a strengthening of the weak interaction condition found in the original sharp propagation of chaos result~\citep{lacker2023hierarchies}. It does not reflect the subsequent improvement in~\cite{ren2024size}, which is able to dispense with this assumption using a non-linear variant of the log-Sobolev inequality. As far as we are aware, the techniques of the latter paper cannot be adapted for our argument.

\begin{assumption}[Very weak interaction]\label{as:weak-interaction}
    Let $\beta_W \bar C_{\LSI} \lesssim 1$ for a small enough implied constant.
\end{assumption}
Broadly speaking, this assumption permits some degree of non-convexity in the problem, although this amount needs to be on a smaller scale than the smoothness parameter of $W$.

\begin{theorem}[Sharp propagation of chaos under very weak interaction]\label{thm:main}
    Suppose Assumptions~\ref{as:smoothness},~\ref{as:isoperimetry} and~\ref{as:weak-interaction} hold, and $N = \widetilde \Omega(1 \vee \sqrt{d} k^{3/2} q)$ for a sufficiently large implied constant. Then,
    \begin{align*}
        \Renyi_q(\mu^{[k]} \mmid \pi^{\otimes k}) \lesssim \Otilde\Bigl(\frac{dk^3 q^2}{N^2}\Bigr)\,.
    \end{align*}
    In particular, for $k = O(1)$, this implies sharp R\'enyi propagation of chaos rates with order $O(N^{-2})$.
\end{theorem}
\begin{remark}
    The scaling $k^3$ is suboptimal and in particular does not recover a sensible result when $k/N \asymp 1$. There are also additional $\log k$ factors hidden by the $\widetilde O$-notation. However, the current proof does not provide an easy avenue to recover the optimal $k^2$ scaling.
\end{remark}

\begin{remark}
    The above also suffices to obtain bounds on $\norm{\frac{\D \mu^{[k]}}{\D \pi^{\otimes k}}}_{L^q(\pi^{\otimes k})} - 1$ when the $q$-R\'enyi is sufficiently small, via a Taylor expansion of the logarithm.
\end{remark}

\begin{remark}
    As a corollary, since it is possible to obtain sampling guarantees in R\'enyi divergence from $\mu^{1:N}$ via many of the results in~\citet{kook2024sampling}, we can also obtain samples $x^1$ which are close in $q$-R\'enyi to $\pi$ with complexities similar to those found in~\citet[Table 1]{kook2024sampling}. 
\end{remark}

We recall a standard change of measure lemma: for any event $A$ and two measures $\mu, \nu$, with $q > 1$,
\begin{align*}
    \mu(A) \leq \nu(A)^{1-\frac{1}{q}} \cdot \exp\Bigl(\frac{q-1}{q} \Renyi_q(\mu \mmid \nu) \Bigr)\,.
\end{align*}
This is a consequence of H\"older's inequality applied to the product of $\one_A$ and $\frac{\D \mu}{\D \nu}$ under the measure $\nu$. Using this, we can deduce the following corollary.
\begin{corollary}[Change of measure]\label{cor:change-measure}
    If $N = \widetilde \Omega(\sqrt{d} k^{3/2})$ and Assumptions~\ref{as:smoothness},~\ref{as:isoperimetry} and~\ref{as:weak-interaction} hold, then for any event $A$,
    \begin{align*}
        \mu^{[k]}(A) \leq 2\bigl[\pi^{\otimes k}(A)\bigr]^{1/2}\,.
    \end{align*}
\end{corollary}
\begin{remark}
    As a consequence, if there is a family of events $\{A_z\}_{z > 0}$ and $\phi: \R_+ \to \R_+$ for which $\pi^{\otimes k}(A_z) \leq \exp(-C \phi(z))$, then $\mu^{[k]}(A_z) \leq 2\exp(-\frac{C}{2} \phi(z))$, so that similar exponential behaviour is observed as $\phi(z) \to \infty$. The constant $C/2$ in the exponent can be improved to $C(1-\zeta)$ for $\zeta > 0$ arbitrarily small by increasing the R\'enyi order $q$. In the work~\cite{jackson2024concentration}, it was shown that, under different assumptions on the potential and interaction kernel, the empirical measure obtained with particles drawn from $\mu^{[N]}$ is $\varepsilon$-close to $\pi$ in Wasserstein-$p$ distances with high probability. 
    
    We can recover this result by first applying exponential concentration results for the empirical measure under $\pi^{\otimes k}$, using techniques from~\cite{fournier2015rate}, and then using Corollary~\ref{cor:change-measure} to guarantee, for $N$ large enough, the same result for the empirical measure under $\mu^{[k]}$. The final result is, up to the constant prefactor and the constant in the exponent, the same behaviour as verified in the setting of~\cite{jackson2024concentration}. Although such a result no longer holds when $k = N$, it is an interesting illustration of the utility of $\Renyi_q$ bounds.
\end{remark}

Before providing a proof of our main result, we first discuss some model examples where the results hold.

\section{Examples}\label{scn:examples}

In this section, we show that the assumptions in this paper hold for the family of examples considered in~\citet{kook2024sampling}.
The following comes from~\citet[Lemma 8]{kook2024sampling}.
\begin{example}[{Perturbations of strong convex potentials}]
    Suppose that the potentials can be decomposed as $V = V_0 + V_1$, $W = W_0 + W_1$, with $\msf{osc}(V_1), \msf{osc}(W_1) < \infty$, where for $U: \R^d \to \R$ we define $\msf{osc}(U) = \sup U - \inf U$. Suppose $V, W$ are $\alpha_V, \alpha_W$ uniformly convex with $\alpha_V + \min(\alpha_W, 0) > 0$. Suppose that $\beta_W \bar C_{\LSI} \lesssim 1$ for a sufficiently small implied constant, where $\bar C_{\LSI}$ is defined below:
    \begin{align*}
        \bar C_{\LSI} \lesssim \frac{1}{\alpha_{V_0} + \frac{N}{N-1} \alpha_{W_0}^-} \exp\Bigl(\msf{osc}(V_1) + \msf{osc}(W_1) \Bigr)\,.
    \end{align*}
    In this case, Assumption~\ref{as:isoperimetry} holds with the same $\bar C_{\LSI}$ as defined above.
\end{example}

\begin{remark}
    This implies as a special case, when $V_1 = W_1 = 0$, that the assumptions are satisfied when the potentials are strongly convex. However, the assumption is slightly more general, permitting very mild perturbations.
\end{remark}

The log-Sobolev constants for $\mu^{1:N}$ and $\pi$ were established in~\cite{kook2024sampling}. It remains to show the log-Sobolev inequality for $\mu^{k+1 \mid [k]}(\cdot \mid x^{[k]})$. Note that the proof of~\citet[Lemma 8]{kook2024sampling} only establishes this for $\mu^{k+1 \mid -(k+1)}$, and that this does not imply the desired property. However proof will be entirely analogous, as we see below.

We have, conditional on $x^{[k]}$, the following form for the measure
\begin{align*}
    \mu^{k+1:N \mid [k]}(x^{k+1:N} \mid x^{[k]}) &\propto \exp\Bigl(-\sum_{i=k+1}^N V(x^i) - \frac{1}{2(N-1)}\sum_{\substack{i,j = k+1 \\
    i \neq j}}^N W(x^i - x^j) \\
    &\hspace{180pt}- \frac{1}{N-1} \sum_{\substack{i = k+1 \\ j \in [k]}}^N W(x^i - x^j) \Bigr)\,.
\end{align*}
If we condition on all the particles except the $i$-th, for $i \geq k+1$, then we have
\begin{align*}
    \mu^{i \mid -i}(\cdot \mid x^{-i}) \propto \exp\Bigl(-V(\cdot) -\frac{1}{N-1} \sum_{j \in [N] \backslash i} W(\cdot - x^j) \Bigr)\,,
\end{align*}
from which it follows via Holley--Stroock that
\begin{align*}
    C_{\LSI}(\mu^{i \mid -i}(\cdot \mid x^{-i})) \leq \frac{1}{\alpha_{V_0} + \frac{N}{N-1} \alpha_{W_0}^-} \exp\Bigl(\msf{osc}(V_1) + \msf{osc}(W_1) \Bigr)\,,
\end{align*}
Conversely, for $i,j \geq k+1$, $i \neq j$, we have
\begin{align*}
    \norm{\nabla_{x^i} \nabla_{x^j} \log \mu^{k+1:N \mid [k]}(\cdot \mid x^{[k]})} \leq \frac{1}{N-1} \norm{\nabla^2 W} \leq \frac{\beta_W}{N-1}\,.
\end{align*}
Thus, appealing to Lemma~\ref{lem:otto-rez} and an identical proof as~\citet[Lemma 8]{kook2024sampling}, we can bound
\begin{align*}
    C_{\LSI}(\mu^{k+1:N \mid [k]}(x^{k+1:N} \mid x^{[k]})) \leq \bar C_{\LSI}\,,
\end{align*}
given in the example. Finally, we note that the log-Sobolev constant is preserved under marginalization, so this suffices to bound the log-Sobolev constants of any conditional measure. Thus, we can verify Assumption~\ref{as:isoperimetry} and Assumption~\ref{as:weak-interaction}.

In the discussion above, we had made use of the following lemma.

\begin{lemma}[{\citet[Theorem 1]{otto2007new}}]\label{lem:otto-rez}
    For a measure $\mu^{1:N}$ on $\R^{d \times N}$, assume that
    \begin{align*}
        C_{\msf{LSI}}(\mu^{i|-i}(\cdot\mid x^{-i})) &\leq \frac{1}{\tau_i}\,, \qquad \text{for all $i \in [N]$, $x^{-i} \in \R^{d \times (N-1)}$}\,, \\
        \norm{\nabla_{x^i} \nabla_{x^j} \log \mu^{1:N}(x^{1:N})} &\leq \beta_{i,j}\,, \qquad \text{for all $x^{1:N} \in \R^{d \times N}$, $i, j \in [N]$, $i \neq j$}\,.
    \end{align*}
    Then, consider the matrix $A \in \R^{N \times N}$ with entries $A_{i,i} = \tau_i$, $A_{i,j} = -\beta_{i,j}$ for $i \neq j$. If $A \succeq \zeta I_{N}$, then $\mu^{1:N}$ satisfies~\eqref{eq:lsi} with constant $1/\zeta$.
\end{lemma}

Since the above permits Gaussian distributions as examples, we perform some explicit calculations to demonstrate where our bounds are tight and where they could yet be improved.
\begin{example}[Explicit computations for Gaussians]
    Consider the following confinement and interaction for $\lambda > 0$,
    \begin{align*}
        V(x) = \frac{1}{2}\norm{x}^2\,, \quad W(x) = \frac{\lambda}{2} \norm{x}^2\,.
    \end{align*}
    Then the stationary distributions are Gaussians, and we can compute $\Renyi_q(\mu^{1:k} \mmid \pi^{\otimes k}) = \widetilde \Omega(\tfrac{dk^2q}{N^2})$ for $N$ sufficiently large, which shows that our upper bound is tight with respect to the $N$ and $d$ scaling, assuming we keep $k, q$ relatively constant.
\end{example}

The same calculations as~\citet[Example 10]{kook2024sampling} give, for $\msf C \in \R^{N \times N}$ with $C_{i,i} = N-1$ and $C_{i,j} = -1$ if $i \neq j$, or in other words $\msf C = N I_N - J$, with $J = \one_N\!\one_N^\top$, $\one_k$ being the length $k$ vector with all entries equal to $1$,
\begin{align*}
    \mu^{1:N} = \mc N\Bigl(0, \bigl(I_{d \times N} + \frac{\lambda}{N-1} \msf C \otimes I_d\bigr)^{-1}\Bigr)\,, \quad \pi = \mc N\Bigl(0, \frac{I_d}{1+\lambda}\Bigr)\,.
\end{align*}
Let us fix $k \ll N$ and write $\rho = 1+\frac{\lambda N}{N-1}$. Then, we can compute that $\mu^{1:k}$, which is a mean-zero Gaussian with covariance given by the upper $(dk \times dk)$ block of the covariance of $\mu^{1:N}$, has the form
\begin{align*}
    \mu^{1:k} = \mc N\Bigl(0, \underset{\Sigma_1}{\underbrace{\Bigl(\frac{1}{\rho} I_k +\frac{\lambda}{\rho(N-1)} \one_k \!\one_k^\top \Bigr) \otimes I_d}} \Bigr)\,.
\end{align*}
On the other hand, $\pi^{\otimes k}$ remains a simple tensor product, with covariance denoted by $\Sigma_2 \deq I_k \otimes \frac{1}{1+\lambda} I_d$.

The R\'enyi divergence with order $q$ between two mean-zero Gaussians is given by
\begin{align*}
    \Renyi_q(\mu^{1:k} \mmid \pi^{\otimes k}) = \frac{1}{2(q-1)} \Bigl[ -q \log\det \Sigma_1 -(1-q) \log \det \Sigma_2 - \log \det \bigl(q\Sigma_1^{-1} +(1-q)\Sigma_2^{-1}\bigr)\Bigr]\,.
\end{align*}
This exists iff $q \Sigma_1^{-1} + (1-q) \Sigma_2^{-1} \succ 0$. From the eigenvalues (noting that $\Sigma_2$ is a multiple of the identity in our case), we require
\begin{align*}
    \frac{\rho}{1+\frac{\lambda k}{N-1}} > \bigl(1-\frac{1}{q} \bigr)\bigl(1+\lambda \bigr)\,.
\end{align*}
After some rearranging, we find that it suffices if
\begin{align*}
    N > 1 + \lambda k(q-1)-\frac{\lambda q}{1+\lambda}\,,
\end{align*}
is sufficiently large. In such a case, noting that $\Sigma_1^{-1}$ has the eigenvalue $\rho$ with multiplicity $d(k-1)$ and $\frac{\rho}{1+\frac{\lambda k}{N-1}}$ with multiplicity $d$, the R\'enyi reduces to
\begin{align*}
    \Renyi_q(\mu^{1:k} \mmid \pi^{\otimes k}) &= \frac{d}{2(q-1)} \Bigl[qk \log \rho -q \log \bigl(1+\frac{\lambda k}{N-1}\bigr) + (1-q) k \log(1+\lambda) \\
    &\hspace{40pt} - (k-1)\log \bigl((1-q)(1+\lambda) + \rho q\bigr) -\log\Bigl((1-q)(1+\lambda)+\frac{\rho q}{1+\frac{\lambda k}{N-1}} \Bigr) \Bigr]\,.
\end{align*}
Now, keeping $\lambda, k, q$ fixed and taking the Taylor expansion of the various functions with respect to $N \to \infty$, some rather tedious computations give, for fixed $q, \lambda$ and $k$,
\begin{align*}
    \Renyi_q(\mu^{1:k} \mmid \pi^{\otimes k}) = \frac{dq\lambda^2}{4(1+\lambda)^2} \cdot \frac{k(k(1+\lambda)^2 - (2\lambda + 1))}{N^2} + O(dN^{-3})\,.
\end{align*}
This shows the $O(d/N^2)$ scaling, although the scaling is only linear in $q$ and quadratic in $k$, when compared with Theorem~\ref{thm:main}.

\section{Proofs}
We start with our main theorem. As the proof will involve several intermediate lemmas, we also provide a flow chart in Figure~\ref{fig:proof-roadmap} for the reader's ease of reference.

\begin{figure}[t]
\centering
\begin{tikzpicture}[
  x=1cm, y=1cm,
  font=\small,
  box/.style={draw, rounded corners, align=left, inner sep=6pt},
  arrow/.style={-Latex, line width=0.6pt},
  darrow/.style={-Latex, dashed, line width=0.6pt}
  group/.style={draw, rounded corners, inner sep=10pt}
]

\node[box] (L5)  at (0,4.3) {\textbf{Lemma~\ref{lem:renyi-lsi}}\\
$\Renyi_q \le \dfrac{q \bar C_{\LSI}}{2}\,\RFI_q$};

\node[box] (DV)  at (0,2.6) {\textbf{Donsker--Varadhan}\\
$\RFI_q \le \textcolor{purple}{\mathrm{(I)}} + \textcolor{teal}{\mathrm{(II)}} + \textcolor{olive}{\mathrm{(III)}}$};

\node[box] (L6)  at (0,0) {\textbf{Lemma~\ref{lem:renyi-fisher-ii}}\\
$\textcolor{teal}{\mathrm{(II)}}\le \dfrac{q^2\bar C_{\LSI}}{2c}\,\RFI_q$\\[2pt]
Choose $c \asymp q^2 \bar C_{\LSI}$.};

\node[box] (L7)  at (-5.2,0) {\textbf{Lemma~\ref{lem:fisher-poc}}\\
$\textcolor{purple}{\mathrm{(I)}}= \widetilde{\mathcal O}\left(\dfrac{dk^3q}{\bar C_{\LSI} N^2}\right)$};

\node[box] (L8)  at (-5.2,2.9) {\textbf{Lemma~\ref{lem:lacker-poc}}\\
KL propagation of chaos};

\node[box] (L9)  at (5.6,0) {\textbf{Lemma~\ref{lem:zeta-subgsn-lacker}}\\
$\textcolor{olive}{\mathrm{(III)}}= \widetilde{\mathcal O}\left(\dfrac{dk^3 q^3}{cN^2} \right)$};

\node[box] (L1011) at (5.6,2.9) {\textbf{Lemmas~\ref{lem:subgsn-conc} and~\ref{lem:lipschitz-recursion}}\\
Lipschitzness/sub-Gaussianity \\
of $W_2\big(\mu^{k+1\mid[k]}(\cdot\mid X^{[k]}),\,\pi\big)$.};

\node[box] (Thm1) at (0,-3.2) {\textbf{Theorem~\ref{thm:main}} \\
$\Renyi_q = \widetilde{\mathcal O}\!\left(\dfrac{dk^3q^2}{N^2}\right)$.};

\draw[arrow] (L5) -- (DV);

\draw[arrow] (DV) -- (L6);

\draw[arrow] (DV) to (L7.north east);
\draw[arrow] (DV) to (L9.north west);

\draw[arrow] (L8) -- (L7);

\draw[arrow] (L6) -- (Thm1);
\draw[arrow] (L7) to (Thm1.north west);
\draw[arrow] (L9) to (Thm1.north east);
\draw[arrow] (L1011.south) -| (L9.north);

\end{tikzpicture}
\caption{Proof roadmap, omitting the arguments in $\Renyi_q$ and $\RFI_q$ for brevity.}
\label{fig:proof-roadmap}
\end{figure}
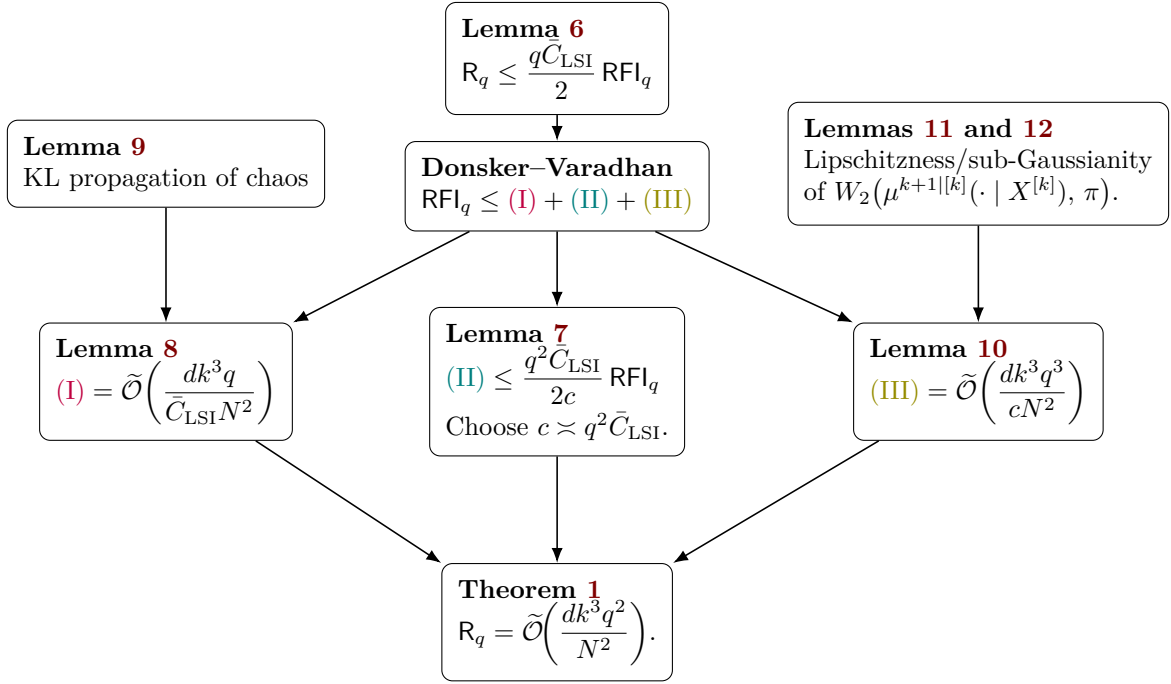

\begin{proof-of-theorem}[{\ref{thm:main}.}]
First, marginalizing over the coordinates $[N] \backslash [k]$, the marginal measures can be presented as follows:
\begin{align*}
    \log \mu^{[k]}(x^{[k]}) \propto \log \int \exp\Bigl(-\sum_{i=1}^k V(x^i) - \frac{1}{2(N-1)}\sum_{\substack{i \in [N] \\
    j \in [N] \backslash i}} W(x^i - x^j)\Bigr) \, \D x^{k+1:N} + \operatorname{const}\,.
\end{align*}
Thus, we have
\begin{align}\label{eq:marginal-score}
\begin{aligned}
    &-\nabla_{x^i} \log \mu^{[k]}(X^{[k]}) \\
    &\qquad = \nabla V(X^i) + \frac{1}{N-1} \sum_{\substack{j \in[k] \\ j \neq i}}\nabla W(X^i - X^j) + \frac{N-k}{N-1} \E_{z \sim \mu^{k+1 \mid [k]}(\cdot \mid X^{[k]})} \nabla W(X^i - z)\,.
\end{aligned}
\end{align}
Introduce the ratios $\rho_k \deq \frac{\mu^{[k]}}{\pi^{\otimes k}}$. Fix an exponent $q \geq 2$. Assuming that $\rho_k^q$ is integrable under $\pi$ (which occurs under the assumptions), then we can define the sequence of tilted expectations and measures as, for any random variable $Z$ on $\R^{d \times k}$ and any Borel set $\msf A \in \mc B(\R^{d \times k})$
\begin{align}\label{eq:twisted-measure}
    \mtt E_k[Z] \deq \E_{\pi^{\otimes k}}\Bigl[\frac{\rho_k^q Z}{\E_{\pi^{\otimes k}}[\rho_k^q]}\Bigr]\,, \qquad \mtt P_k(\cdot \in \msf A) = \int_{\msf A} \frac{\rho_k^q(\cdot)}{\E_{\pi^{\otimes k}}[\rho_k^q]} \, \D \pi^{\otimes k}(\cdot)\,.
\end{align}
Assuming $\pi$ satisfies~\eqref{eq:lsi} with constant $\bar C_{\LSI}$, we have via Lemma~\ref{lem:renyi-lsi}
\begin{align*}
    \Renyi_q(\mu^{[k]} \mmid \pi^{\otimes k}) &\leq \frac{q\bar C_{\LSI}}{2}  \RFI_q(\mu^{[k]} \mmid \pi^{\otimes k}) = \frac{q^2 \bar C_{\LSI}}{2} \mtt E_k[\norm{\nabla \log \rho_k}^2]\,.
\end{align*}
Introduce the variable
\begin{align}\label{eq:zeta-def}
    \zeta_k \deq c(\norm{\nabla \log \rho_k} -\E_{\mu^{[k]}}\norm{\nabla \log \rho_k})^2\,,
\end{align}
for a parameter $c > 0$. We now have, via the Donsker--Varadhan variational principle applied to the test variable $\zeta_k$,
\begin{align*}
    \mtt E_k[\norm{\nabla \log \rho_k}^2] &\leq 2 \E_{\mu^{[k]}}[\norm{\nabla \log \rho_k}^2] + \frac{2}{c} \mtt E_k[\zeta_k] \\
    &\leq 2\FI(\mu^{[k]} \mmid \pi^{\otimes k}) + \frac{2}{c}\Bigl\{\KL(\mtt P_k \mmid \mu^{[k]}) + \log \E_{\mu^{[k]}} \exp(\zeta_k) \Bigr\}\,.
\end{align*}
Using Lemma~\ref{lem:renyi-fisher-ii} we can bound again
\begin{align*}
    \KL(\mtt P_k \mmid \mu^{[k]}) \leq \frac{(q-1) \bar C_{\LSI}}{2} \RFI_q(\mu^{[k]} \mmid \pi^{\otimes k})\,.
\end{align*}
Choose $c = 2q^2 \bar C_{\LSI}$. We can now bound
\begin{align*}
    \RFI_q(\mu^{[k]} \mmid \pi^{\otimes k}) \leq \frac{1}{2} \RFI_q(\mu^{[k]} \mmid \pi^{\otimes k}) + 2q \FI(\mu^{[k]} \mmid \pi^{\otimes k}) + \frac{1}{q \bar C_{\LSI}} \log \E \exp(\zeta_k)\,.
\end{align*}
Rearranging, we have
\begin{align*}
    \RFI_q(\mu^{[k]} \mmid \pi^{\otimes k}) \leq 4 q \FI(\mu^{[k]} \mmid \pi^{\otimes k}) + \frac{2}{q \bar C_{\LSI}} \log \E \exp(\zeta_k)\,.
\end{align*}
These terms are controlled through Lemma~\ref{lem:fisher-poc} and Lemma~\ref{lem:zeta-subgsn-lacker} respectively. This gives our final result, which is
\begin{align*}
    \Renyi_q(\mu^{[k]} \mmid \pi^{\otimes k}) \lesssim \Otilde\Bigl(\frac{dk^3 q^2}{N^2}\Bigr)\,.
\end{align*}
\end{proof-of-theorem}

We now prove the intermediate results used to develop the main theorem. The next two lemmas largely follow from elementary calculations.
\begin{lemma}[{Adapted from~\citet[Lemma 5]{VempalaW19}}]\label{lem:renyi-lsi}
    If a measure $\nu$ satisfies~\eqref{eq:lsi} with constant $C_{\LSI}$, then for any $\mu \in \mc P(\R^d)$,
    \begin{align*}
        \Renyi_q(\mu \mmid \nu) \leq \frac{q C_{\LSI}}{2} \RFI_q(\mu \mmid \nu)\,.
    \end{align*}
\end{lemma}
\begin{proof}
The following derivation comes from~\citet[Theorem 2.2.25]{chewibook}. We have, applying~\eqref{eq:lsi} to the test function $f = \rho^{q/2}$, for $\rho = \frac{\D \mu}{\D \nu}$ with $\mathscr E(f, f) = \E_\nu \norm{\nabla f}^2$ the Dirichlet form corresponding to $\nu$,
\begin{align*}
    2 C_{\LSI} \mathscr E(\rho^{q/2}, \rho^{q/2}) &\geq q \int \rho^q \log \rho \, \D \nu - \Bigl(\int \rho^q \, \D \nu \Bigr) \log \Bigl(\int \rho^q \, \D \nu \Bigr) \\
    &= q \partial_q \int \rho^q \, \D \nu - \Bigl( \int \rho^q \, \D \nu\Bigr)\log \Bigl(\int \rho^q \, \D \nu \Bigr) \,. 
\end{align*}
Thus,
\begin{align*}
    \frac{4}{q} \frac{\mathscr E(\rho^{q/2}, \rho^{q/2})}{\int \rho^q \, \D \nu} &\geq \frac{2}{C_{\LSI}} \Bigl\{ \partial_q \log \int \rho^q \, \D \nu - \frac{1}{q}\log \Bigl(\int \rho^q \, \D \nu \Bigr)\Bigr\} \\
    &= \frac{2}{C_{\LSI}} \partial_q [(q-1) \Renyi_q(\mu \mmid \nu)] - \frac{2(q-1)}{q C_{\LSI}} \Renyi_q(\mu \mmid \nu) \\
    &=  \frac{2}{q C_{\LSI}} \Renyi_q(\mu \mmid \nu) + \frac{2(q-1)}{C_{\LSI}} \partial_q \Renyi_q(\mu \mmid \nu)\,.
\end{align*}
The conclusion follows, noting that $\Renyi_q(\cdot \mmid \cdot)$ is increasing in $q$.
\end{proof}
\begin{lemma}[{Adapted from~\citet[Proof of Lemma 19]{chewi2025analysis}}]\label{lem:renyi-fisher-ii}
    If $\pi^{\otimes k}$ satisfies a log-Sobolev inequality with constant $\bar C_{\LSI}$, we have for $q \geq 2$ and any $\mu^{[k]} \in \mc P(\R^{d \times k})$, with $\mtt P_k$ defined as in~\eqref{eq:twisted-measure},
     \begin{align*}
        \KL(\mtt P_k \mmid \mu^{[k]}) \leq \frac{(q-1) \bar C_{\LSI}}{2} \RFI_q(\mu^{[k]} \mmid \pi^{\otimes k})\,.
    \end{align*}
\end{lemma}
\begin{proof}
    This follows the derivation in~\cite[Theorem 6.1.2]{chewibook}. Let $\psi_k^q = \frac{\rho_k^{q-1}}{\E_{\pi^{\otimes k}}[\rho_k^{q}]}$. Then
    \begin{align*}
        \KL(\mathtt P_k \mmid \mu^{[k]}) &= \E_{\psi_k^q \mu^{[k]}} \log \psi_k^q = \E_{\psi_k^q \mu^{[k]}} \log \frac{\rho_k^{q-1}}{\E_{\mu^{[k]}} \rho_k^{q-1}} = \frac{q-1}{q} \E_{\psi_k^q \mu^{[k]}} \log \frac{\rho_k^{q}}{\E_{\mu^{[k]}}[\rho_k^{q-1}]^{q/(q-1)}} \\
        &=\frac{q-1}{q}\Bigl\{\E_{\psi_k^q \mu^{[k]}} \log \frac{\rho_k^q}{\E_{\mu^{[k]}} \rho_k^{q-1}} -\underset{(*)}{\underbrace{\frac{1}{q-1} \log \E_{\mu^{[k]}}[\rho_k^{q-1}]}}\Bigr\} \\
        &\leq \frac{q-1}{q} \KL(\psi_k^q \mu^{[k]} \mmid \pi^{\otimes k}) \\
        &\leq \frac{(q-1) \bar C_{\LSI}}{2q} \E_{\psi_k^q \mu^{[k]}}\bigl[ \norm{\nabla \log \rho_k^q}^2\bigr] = \frac{(q-1) \bar C_{\LSI}}{2} \RFI_q(\mu^{[k]} \mmid \pi^{\otimes k})\,.
    \end{align*}
    The term $(*) \geq 0$ as it is equal to $\Renyi_q(\mu^{[k]} \mmid \pi^{\otimes k})$.
\end{proof}

The next lemma, for bounding the Fisher information, is adapted from the intermediate steps of the proof of~\citet[Theorem 3]{kook2024sampling}. Unlike the proof in said paper, we will not be concerned about absolute constants, as we will not need to derive a recursive structure. Instead we will rely upon~\citet[Theorem 3]{kook2024sampling} to provide the propagation of chaos results.
\begin{lemma}\label{lem:fisher-poc}
    Under Assumptions~\ref{as:smoothness},~\ref{as:isoperimetry} and \ref{as:weak-interaction}, for $N \gtrsim 1$ sufficiently large and $k \in [N]$, it holds that
    \begin{align*}
        \bar C_{\LSI} \FI(\mu^{[k]} \mmid \pi^{\otimes k}) = \widetilde O\Bigl( \frac{d k^3}{N^2}\Bigr)\,.
    \end{align*}
\end{lemma}
\begin{proof}
Using the representations given in~\eqref{eq:marginal-score} and~\eqref{eq:mf-minimizer}, we derive
\begin{align*}
    \bar C_{\LSI}\FI(\mu^{[k]} \mmid \pi^{\otimes k}) &= \bar C_{\LSI} \sum_{i=1}^k \E_{\mu^{[k]}}\bigl[\bigl \lVert \frac{1}{N-1} \sum_{\substack{j=1\\ j \neq i}}^k \nabla W(X^i -X^j) - \int \nabla W(X^i - \cdot) \, \D \pi \\
    &\hspace{90pt} + \frac{N-k}{N-1} \int \nabla W(X^i - \cdot) \, \D \mu^{k+1 \mid [k]}(\cdot \mid X^{[k]})\bigr\rVert^2\bigr]\\
    &\leq \frac{2k \bar C_{\LSI}}{(N-1)^2} \underset{\msf A}{\underbrace{\E_{\mu^{[k]}}\bigl[\norm{{\sum_{j=2}^k} \nabla W(X^1 - X^j) - \int \nabla W(X^1 - \cdot)\, \D \pi}^2\bigr]}} \\
    &\hspace{40pt} + \frac{2k \bar C_{\LSI} (N-k)^2}{(N-1)^2}\underset{\msf B}{\underbrace{\E_{\mu^{[k]}}\bigl[\norm{\int \nabla W(X^1 - \cdot)\Bigl(\, \D \mu^{k+1 \mid [k]}(\cdot \mid X^{[k]}) - \, \D \pi \Bigr)}^2\bigr]}}\,,
\end{align*}
where the second line follows from exchangeability.
We now handle terms $\msf A, \msf B$ separately.
\begin{align*}
    \msf{A} 
    &= \sum_{j=2}^k \E_{\mu^{[k]}}[\norm{\nabla W(X^1 - X^j) - \E_\pi\nabla W(X^1 - \cdot)}^2] \\
    &\qquad + \sum_{\substack{i,j=2 \\i \neq j}}^k \E_{\mu^{[k]}} \bigl\langle \nabla W(X^1 - X^i) - \E_\pi \nabla W(X^1 - \cdot),\,\nabla W(X^1 - X^j) - \E_{\pi}\nabla W(X^1 - \cdot)\bigr\rangle \\
    &\underset{\text{(i)}}{=} (k-1) \E[\norm{\nabla W(X^1 - X^2) - \E_{ \pi}\nabla W(X^1 - \cdot)}^2] \\
    &\qquad +(k-1)\,(k-2)\E\bigl \langle \nabla W(X^1- X^2) - \E_{\pi} \nabla W(X^1 - \cdot),\\
    &\qquad\qquad\qquad\qquad\qquad \qquad\qquad\qquad\qquad\qquad\qquad\nabla W(X^1 - X^3) - \E_{\pi}\nabla W(X^1 - \cdot)\bigr\rangle \\
    &\underset{\text{(ii)}}{\leq} k\, \beta_W^2 \E[\norm{X - Y}^2] \\
    &\quad + k^2 \E\bigl \langle \nabla W(X^1- X^2) - \E_{\pi} \nabla W(X^1 - \cdot),\,
    \nabla W(X^1 - X^3) - \E_{\pi}\nabla W(X^1 - \cdot)\bigr\rangle\,.
\end{align*}
(i) uses exchangeability of the particles and (ii) uses smoothness of $W$. Here, $X\sim \mu^1$ and $Y\sim \pi$ are independent. We handle these two terms separately.

Optimally couple $Z \sim \pi$ and $X$ in the first term, so that by independence and sub-Gaussian concentration (implied by~\eqref{eq:lsi}),
\begin{align}
    \E[\norm{X-Y}^2]
    &\lesssim \E[\norm{X-Z}^2] + \E[\norm{Y-Z}^2]
    \leq \,W_2^2(\mu^1,\pi) + 4\E[\norm{Y-\E Y}^2] \nonumber\\
    &\lesssim  \bar C_{\LSI} \KL(\mu^1 \mmid \pi) + d \bar C_{\LSI}\,,\label{eq:second_moment_bd}
\end{align}
where the second inequality follows from \eqref{eq:talagrand}, and the last one follows from the data-processing inequality for the $\KL$ divergence.

Applying the Cauchy--Schwarz inequality to the second leads to
\begin{align}
    &\E\bigl \langle \nabla W(X^1- X^2) - \E_{\pi} \nabla W(X^1 - \cdot),\, \nabla W(X^1 - X^3) - \E_{\pi}\nabla W(X^1 - \cdot)\bigr\rangle \nonumber\\
    &\qquad = \E\bigl \langle \nabla W(X^1- X^2) - \E_{\pi} \nabla W(X^1 - \cdot),\,\E_{\mu^{3 \mid [2]}(\cdot \mid X^{[2]})} \nabla W(X^1 - \cdot) - \E_{\pi}\nabla W(X^1 - \cdot)\bigr\rangle \nonumber\\
 & \qquad\leq\beta_{W}^{2}\sqrt{\E[\norm{X-Y}^{2}]}\sqrt{\E W_2^2\bigl(\mu^{3\mid [2]}(\cdot \mid X^{[2]}),\, \pi\bigr)}\nonumber \\
 & \qquad\underset{\text{(i)}}{\lesssim}\beta_{W}^{2}\sqrt{\bar C_{\msf{LSI}}\,(\KL_{3}+d)}\sqrt{\bar C_{\msf{LSI}}\E\KL\bigl(\mu^{3\mid [2]}(\cdot \mid X^{[2]}) \bigm\Vert \pi\bigr)}\nonumber \\
 & \qquad\underset{\text{(ii)}}{\lesssim}\beta_{W}^{2}\bar C_{\msf{LSI}}\sqrt{\KL_{3}+d}\sqrt{\KL_{3}}\lesssim \beta_{W}^{2}\bar C_{\msf{LSI}}\,(\KL_{3}+d)\,,\nonumber 
\end{align}
where in (i) we applied the bound~\eqref{eq:second_moment_bd} as well as~\eqref{eq:talagrand}, and in (ii) we used the chain rule for the $\msf{KL}$ divergence, dropping the negative term. Here, $\KL_k$ is a shorthand for $\KL(\mu^{[k]} \mid \pi^{\otimes k})$.

We return to the analysis of the term $\msf B$. We obtain 
\begin{align*}
\msf B
&= \E\Bigl[\Bigl\lVert\int \nabla W(X^1 - \cdot) \,\bigl(\D \mu^{k+1|[k]}(\cdot \mid X^{[k]}) - \D \pi\bigr)\Bigr\rVert^2\Bigr]
\le \beta_{W}^{2}\E W_{2}^{2}\bigl(\mu^{k+1|[k]}(\cdot\mid X^{[k]}),\,\pi\bigr)\\
 & \lesssim \beta_{W}^{2}\bar C_{\msf{LSI}}\KL_{k+1}\,.
\end{align*}
where the second line follows in a similar way as for term $\msf A$.
Putting our bounds on $\msf A$ and $\msf B$ together, we obtain for $N\geq 30$,
\begin{align}\label{eq:recursive-ineq}
\msf{FI}(\mu^{[k]} \mmid \pi^{\otimes k}) \lesssim \frac{ k^{3}\beta_{W}^{2}\bar C_{\msf{LSI}}^2}{N^2}\,(\KL_{3}+d)+{k\beta_{W}^{2}\bar C_{\msf{LSI}}^2}\KL_{k+1}\,.
\end{align}
The result follows by appealing to Lemma~\ref{lem:lacker-poc} and using Assumption~\ref{as:weak-interaction}.
\end{proof}

We again remark that the extra factor of $k$ can be removed when attempting the propagation of chaos argument directly for $\KL$ via a clever bootstrap argument, as seen in Lemma~\ref{lem:lacker-poc} below. However, for $\FI$, this type of argument does not appear to easily apply.

In the proof of Lemma~\ref{lem:fisher-poc}, we made use of the sharp propagation of chaos in $\KL$ divergence, which originally stems from the analysis of~\citet{lacker2023hierarchies, lacker2023sharp}, and is given without proof below:
\begin{lemma}[{$\KL$-propagation of chaos; adapted from~\citet[Theorem 3]{kook2024sampling}}]\label{lem:lacker-poc}
    Under Assumptions~\ref{as:smoothness},~\ref{as:isoperimetry} and \ref{as:weak-interaction}, for $N \gtrsim 1$ sufficiently large and $k \in [N]$, it holds that
    \begin{align*}
        \KL(\mu^{[k]} \mmid \pi^{\otimes k}) = \widetilde O\Bigl(\frac{dk^2}{N^2} \Bigr)\,.
    \end{align*}
\end{lemma}
Finally, the following lemma is used to control the exponential term arising from Donsker--Varadhan, and contains the most involved proof. It will largely be reliant on the Lipschitz hierarchy lemma, which we prove in the subsequent section.
\begin{lemma}[Exponential moment of $\zeta_k$]\label{lem:zeta-subgsn-lacker}
    Under Assumptions~\ref{as:smoothness},~\ref{as:isoperimetry} and~\ref{as:weak-interaction}, recalling the definition of $\zeta_k$ in~\eqref{eq:zeta-def}, with $c \asymp q^2 \bar C_{\LSI}$ it follows that
    \begin{align*}
        \log \E_{\mu^{[k]}} \exp(\zeta_k) = \Otilde\Bigl(\frac{dk^3 q^2}{N^2} \Bigr)\,,
    \end{align*}
    so long as $N = \widetilde{\Omega}(\sqrt{d} k^{3/2} q)$ for a sufficiently large implied constant.
\end{lemma}
\begin{proof}
First, we recall that $\zeta_k \deq c\bigl(\norm{\nabla \log \rho_k} - \E_{\mu^{[k]}} \norm{\nabla \log \rho_k} \bigr)^2$, with $c \asymp q^2 \bar C_{\LSI}$.

The primary lemma in our proof is a concentration result for $\zeta_k$ based on Lipschitzness. We will simply work with the uncentred quantity $\norm{\nabla \log \rho_k(x^{[k]})}$. First, we have for a fixed $i \in [k]$,
\begin{align*}
    \nabla_{x^i} \log \rho_k(x^{[k]}) &= \frac{1}{N-1} \underset{\msf A_{i,k}}{\underbrace{ \sum_{\substack{j=1\\
    j \neq i}}^k \Bigl(\nabla W(x^i - x^j) - \int \nabla W(x^i - \cdot)\, \D \pi\Bigr)}}\\
    &\qquad + \frac{N-k}{N-1} \underset{\widetilde{\msf B}_{i,k}}{\underbrace{\Bigl(\int \nabla W(x^i - \cdot)\, \bigl(\D \mu^{(k+1)\mid [k]}(\cdot \mid x^{[k]}) - \D \pi\bigr)\Bigr)}}\,.
\end{align*}
We will show Lipschitzness of $\msf A_{i,k}$ and $\msf B_{i,k}$, viewing both implicitly as functions on $\R^d$ of $x^{-i}$. Now, it follows via Lipschitzness of $W$ that for all $x^{[k]}, \bar x^{[k]} \in \R^{d \times k}$,
\begin{align}\label{eq:term-A-lip}
    \norm{\msf A_{i,k}(x^{[k]}) - \msf A_{i,k}(\bar x^{[k]})}^2 \lesssim \beta_W^2 (k\norm{x^{-i} - \bar x^{-i}}^2 + k^2 \norm{x^i - \bar x^i}^2)\,.
\end{align}
Now, let us turn our attention to $\msf B_{i,k}$, which will take some more effort. While it is not easy to directly prove Lipschitzness of $\msf B_{i,k}$ with a good constant, it will be much easier to obtain Lipschitzness of an upper bound on its norm. Using the smoothness of $W$ we have
\begin{align*}
    \norm{\widetilde{\msf B}_{i,k}(x^{[k]})}^2 \leq \beta_W^2 W_2^2(\mu^{k+1 \mid [k]}(\cdot \mid x^{[k]}), \pi) \eqqcolon \msf B_{i,k}^2 \,.
\end{align*}
As we shall see, the right side itself is the square of a Lipschitz function of the variables $x^{[k]}$. We show this through a hierarchical argument reminiscent of the $\KL$ proof in~\cite{lacker2023hierarchies}. See Lemma~\ref{lem:lipschitz-recursion}, which we defer to the subsequent section.

Finally, we show how all of these imply our concentration results. First, we can bound
\begin{align*}
    \Bigl(\norm{\nabla \log \rho_k} - \E_{\mu^{[k]}} \norm{\nabla \log \rho_k} \Bigr)^2 &\lesssim \norm{\nabla \log \rho_k}^2 + \E_{\mu^{[k]}}[\norm{\nabla \log \rho_k}^2]\\
    &\lesssim \sum_{i=1}^k \Bigl(\frac{1}{N^2}\norm{\msf A_{i,k}}^2 + \norm{\msf B_{i,k}}^2 \Bigr) + \E_{\mu^{[k]}}[\norm{\nabla \log \rho_k}^2] \\
    &\lesssim \sum_{i=1}^k \underset{\msf S_i}{\underbrace{\Bigl(\Bigl\lvert\frac{\norm{\msf A_{i,k}}}{N} - \frac{\E_{\mu^{[k]}}[\norm{\msf A_{i,k}}]}{N} + \norm{\msf B_{i,k}}-\E_{\mu^{[k]}}[\norm{\msf B_{i,k}}]\Bigr\rvert^2 \Bigr)}} \\
    &\qquad+ \underset{\msf T}{\underbrace{\sum_{i=1}^k \Bigl(\frac{1}{N^2}\E_{\mu^{[k]}}[\norm{\msf A_{i,k}}^2] + \E_{\mu^{[k]}}[\norm{\msf B_{i,k}}^2] \Bigr) + \E_{\mu^{[k]}}[\norm{\nabla \log \rho_k}^2]}}\,.
\end{align*}
Now, we evaluate, noting that $\msf T$ is deterministic, for some absolute constant $C$,
\begin{align*}
    \E_{\mu^{[k]}} \exp\bigl(\zeta_k\bigr) \lesssim \exp(cC \msf T) \cdot \E_{\mu^{[k]}} \exp\bigl(c C \sum_{i \in [k]} \msf S_i\bigr) \lesssim \exp(cC \msf T) \cdot\max_{i \in [k]}  \E_{\mu^{[k]}} \exp\bigl(c k C\msf S_i\bigr)\,,
\end{align*}
where we applied AM--GM. Now, for all $i \in [k]$, $\norm{\msf B_{i,k}}$ is $\frac{\beta_W \sqrt{k}}{N}$ Lipschitz via Lemma~\ref{lem:lipschitz-recursion} while $\msf A_{i,k}$ is $\beta_W k$ Lipschitz. Thus, each $\msf S_i$ has the form of the square of a centered random variable with Lipschitz constant $\frac{\beta_W k}{N}$. As $\mu^{[k]}$ itself satisfies LSI with constant $\bar C_{\LSI}$, so $\msf S_i$ is sub-Gaussian with variance proxy $\frac{\beta_W^2 k^2 \bar C_{\LSI}}{N^2}$, using Lemma~\ref{lem:subgsn-conc}. As a result, using Lemma~\ref{lem:subgsn-conc}
\begin{align*}
    \log \max_{i \in [k]} \E_{\mu^{[k]}} \exp\bigl(c k C\msf S_i\bigr) \lesssim \frac{\beta_W^2 ck^3 \bar C_{\LSI} }{N^2}\,,
\end{align*}
so long as $\frac{\beta_W^2 ck^3 \bar C_{\LSI} }{N^2} \lesssim 1$ for a sufficiently small implied constant. Recalling that $c \asymp q^2 \bar C_{\LSI}$, and using Assumption~\ref{as:weak-interaction}, we can bound this by $\widetilde O(\frac{k^3 q^2}{N^2})$ as claimed.

Following the proof of Lemma~\ref{lem:fisher-poc}, we can also bound
\begin{align*}
    &\bar C_{\LSI} \sum_{i=1}^k \Bigl(\frac{1}{N^2} \E_{\mu^{[k]}} [\norm{\msf A_{i,k}}^2] + \E_{\mu^{[k]}}[\norm{\msf{B}_{i,k}}^2]\Bigr) = \Otilde\Bigl(\frac{dk^3}{N^2}\Bigr) \qquad
    \bar C_{\LSI}\Bigl(\E_{\mu^{[k]}}[\norm{\nabla \log \rho_k}^2]\Bigr) = \Otilde\Bigl(\frac{dk^3}{N^2}\Bigr)\,.
\end{align*}
This immediately gives a bound on the $\msf T$-term. Putting all this together concludes the proof.
\end{proof}

In the proof above, we made use of some standard sub-Gaussian concentration properties. See e.g.,~\cite{boucheron2003concentration} for proofs.
\begin{lemma}\label{lem:subgsn-conc}
    If $G: \R^d \to \R$ is $L$-Lipschitz, and $\mu \in \mc P(\R^d)$ satisfies~\eqref{eq:lsi} with constant $C_{\LSI}$, then
    \begin{align*}
        \log \E_\mu \exp\bigl(\lambda(G- \E_\mu G) \bigr) \leq \frac{1}{2} \lambda^2 L^2 C_{\LSI}\,.
    \end{align*}
    In particular, this implies that the variance proxy can be bounded by $C_{\LSI} L^2$, and thus that
    \begin{align*}
        \log \E_\mu e^{\lambda (G - \E_\mu G)^2} \leq 2 \lambda L^2 C_{\LSI} \,, \quad \mathrm{if }\quad  \lambda L^2 C_{\LSI} \leq \frac{1}{4}\,.
    \end{align*}
\end{lemma}

\subsection{Lipschitz recursion}
\begin{lemma}\label{lem:lipschitz-recursion}
    Under Assumption~\ref{as:weak-interaction}, $N \geq 100$, $k \ll N$, it holds for any $x^{[k]}, \bar x^{[k]} \in \R^{d \times k}$ that
    \begin{align*}
        \abs{W_2(\mu^{k+1 \mid [k]}(\cdot \mid x^{[k]}) , \pi) - W_2(\mu^{k+1 \mid [k]}(\cdot \mid \bar x^{[k]}), \pi)} \lesssim \frac{\sqrt{k} \norm{x^{[k]} - \bar x^{[k]}}}{N}\,.
    \end{align*}
\end{lemma}
\begin{proof}
    Applying the triangle inequality then~\eqref{eq:talagrand} and then~\eqref{eq:lsi}, we find that, letting $\upmu^{\ell}, \bar \upmu^{\ell}$ denote the measures $\mu^{k+1:k+\ell \mid [k]}(\cdot \mid x^{[k]}), \mu^{k+1:k+\ell \mid [k]}(\cdot \mid \bar x^{[k]})$ respectively for $\ell \in [N-k]$,
    \begin{align*}
        \abs{W_2(\upmu^\ell,\pi) - W_2(\bar \upmu^\ell(\cdot), \pi)} \leq \abs{W_2(\upmu^\ell, \bar \upmu^\ell)}
        \leq \sqrt{2\bar C_{\LSI}} \sqrt{\KL\bigl(\upmu^\ell \mid \bar \upmu^\ell\bigr)} \leq \bar C_{\LSI} \sqrt{\FI\bigl(\upmu^\ell \mid \bar \upmu^\ell\bigr)}\,.
    \end{align*}
    Now, we have for $i \in [\ell]$,
    \begin{align*}
        \nabla_{z^i} \log \upmu^\ell(z^{1:\ell}) &= -\nabla V(z^i) - \frac{1}{N-1} \sum_{j=1}^k \nabla W(z^i-x^j) - \frac{1}{N-1} \sum_{\substack{j'=1 \\ j \neq j'}}^\ell \nabla W(z^i - z^{j'}) \\
        &\qquad- \frac{N-k-\ell}{N-1} \int \nabla W(z^i-\cdot) \, \D \mu^{k+\ell+1 \mid [k]+\ell}(\cdot \mid x^{[k]}, z^{1:\ell})\,.
    \end{align*}
    As a result, we can proceed with, defining $\uprho^\ell(z^{1:\ell}) = \frac{\upmu^\ell}{\bar \upmu^\ell}(z^{1:\ell})$, using exchangeability of $\upmu^\ell$ in its $\ell$ coordinates,
    \begin{align*}
        \E_{\upmu^\ell}&[\norm{\nabla \log \uprho^\ell}^2] \\
        &\leq \frac{2\ell}{(N-1)^2}\, \underset{\msf C_\ell}{ \underbrace{\E_{\upmu^\ell}\Bigl[\norm{\sum_{j=1}^k (\nabla W(z^1-x^j) - \nabla W(z^1-\bar x^j))}^2 \Bigr]}} \\
        &\qquad+ \frac{2\ell (N-k-\ell)^2}{(N-1)^2}\, \underset{\msf D_\ell}{\underbrace{\E_{z^{[\ell]} \sim \upmu^\ell}\Bigl[\norm{\int \nabla W(z^1-\cdot)\{ \D \mu^{k+\ell + 1 \mid [k+\ell]} - \D \bar \mu^{k+\ell+1 \mid [k+\ell]}\} }^2 \Bigr]}}\,,
    \end{align*}
    where we defined $\mu^{k+\ell+1 \mid [k+\ell]} = \mu^{k+\ell+1 \mid [k+\ell]}(\cdot \mid x^{[k]}, z^{[\ell]})$ and $\bar \mu^{k+\ell+1 \mid [k+\ell]} = \mu^{k+\ell+1 \mid [k+\ell]}(\cdot \mid \bar x^{[k]}, z^{[\ell]})$ as shorthands.
    For the first term, simply applying Lipschitzness of $\nabla W$ and Cauchy--Schwarz permits us to bound it by
    \begin{align*}
        \msf C_\ell \leq \beta_W^2 k \norm{x^{[k]} - \bar x^{[k]}}^2\,.
    \end{align*}
    On the other hand, the chain-rule for $\KL$ and Assumption~\ref{as:isoperimetry} yield
    \begin{align*}
        \msf D_\ell &\leq \beta_W^2 \E_{z^{[\ell]} \sim \upmu^\ell} W_2^2(\mu^{k+\ell+1\mid [k+\ell]}(\cdot \mid x^{[k]}, z^{[\ell]}), \mu^{k+\ell+1 \mid [k+\ell]}(\cdot \mid \bar x^{[k]}, z^{[\ell]})) \\
        &\leq 2\beta_W^2 \bar C_{\LSI} \E_{z^{[\ell]} \sim \upmu^\ell} \KL(\mu^{k+\ell+1 \mid [k+\ell]}(\cdot \mid x^{[k]}, z^{[\ell]}) \mmid \mu^{k+\ell+1 \mid [k+\ell]}(\cdot \mid \bar x^{[k]}, z^{[\ell]})) \\
        &= 2\beta_W^2 \bar C_{\LSI}\Bigl\{\KL(\upmu^{\ell+1} \mmid \bar \upmu^{\ell+1}) - \KL(\upmu^\ell \mmid \bar \upmu^\ell)\Bigr\}\,.
    \end{align*}
    We can now rearrange this inequality. Letting $\msf K_\ell = \KL(\upmu^{\ell} \mid \bar \upmu^\ell)$, we have shown the following hierarchy,
    \begin{align*}
        \msf K_\ell \leq \underset{\mc C_\ell}{\underbrace{\frac{2\ell \beta_W^2 \bar C_{\LSI}^2}{1 + 2\ell \beta_W^2 \bar C_{\LSI}^2}}}\, \Bigl(\frac{k}{2\bar C_{\LSI} N^2} \norm{x^{[k]} - \bar x^{[k]}}^2 + \msf K_{\ell+1}  \Bigr)\,.
    \end{align*}
    Finally, at $\msf K_{N-k}$, we bound (as the term $\msf D$ is not present in this case)
    \begin{align*}
        \msf K_{N-k} \leq \frac{\beta_W^2 k \bar C_{\LSI}}{2N} \norm{x^{[k]} - \bar x^{[k]}}^2\,.
    \end{align*}
    Solving this recurrence relation, we obtain, defining the ratio $\xi \deq (2\beta_W^2 \bar C_{\LSI}^2)^{-1}$, 
    \begin{align*}
        \msf K_1 &\leq \biggl\{\Bigl(\prod_{\ell=1}^{N-k-1} C_\ell \Bigr) \frac{\beta_W^2 k \bar C_{\LSI}}{2N} + \sum_{\ell=1}^{N-k-1} \Bigl(\prod_{\ell' = 1}^\ell C_{\ell'}\Bigr) \frac{k}{2\bar C_{\LSI} N^2}\, \biggr\}\norm{x^{[k]} - \bar x^{[k]}}^2 \\
        &\leq \underset{c_N}{\underbrace{\biggl\{\Bigl(\prod_{\ell=1}^{N-k-1} C_\ell \Bigr) \frac{1}{4\xi N} + \sum_{\ell=1}^{N-k-1} \Bigl(\prod_{\ell' = 1}^\ell C_{\ell'}\Bigr) \frac{1}{N^2}\, \biggr\}}} \frac{k\norm{x^{[k]} - \bar x^{[k]}}^2}{2\bar C_{\LSI}} \,.
    \end{align*}
    Now, we use the generic formula in Lemma~\ref{lem:coeff-product} for $k \ll N$,
    \begin{align*}
        \msf K_1 \leq (1+\xi)^\xi \Bigl(\frac{N^{-1-\xi}}{\xi} + \frac{1}{N^2} \sum_{i=1}^{N-k-1} \frac{1}{i^\xi} \Bigr)\frac{k \norm{x^{[k]} - \bar x^{[k]}}^2}{2 \bar C_{\LSI}}\,.
    \end{align*}
    Of course, under our assumptions, we already know that $\xi \geq 2$, and we can always bound it above by an absolute constant by taking a worse bound on the $\beta_W$. As a result, we can bound $\sum_{i \leq N-k-1} i^{-\xi} \lesssim 1$, which suffices to bound this overall by
    \begin{align*}
        \msf K_1 \lesssim  \frac{k \norm{x^{[k]} - \bar x^{[k]}}^2}{\bar C_{\LSI} N^2}\,.
    \end{align*}
\end{proof}

Finally, we made use of the following standard lemma to control the coefficients.
\begin{lemma}[{Adapted from~\citet[Lemma 15]{kook2024sampling}}]\label{lem:coeff-product}
    For $1 \leq i \leq j \leq N-k$, with $\xi \deq \bigl(2 \beta_W^2 \bar C^2_{\LSI} \bigr)^{-1}$, we have
    \begin{align*}
        \prod_{\ell = i}^j C_\ell \leq \Bigl(\frac{i+\xi}{j+1 +\xi} \Bigr)^\xi\,.
    \end{align*}
\end{lemma}

\section{Discussion}

We briefly outline some directions for future work. Firstly, the argument above has not been made dynamic, i.e., for the simultaneous evolutions of a measure along~\eqref{eq:mckean-vlasov} and~\eqref{eq:finite-particle}. This is due to difficulties in propagating Assumption~\ref{as:isoperimetry} along the dynamics, as the stability of the log-Sobolev constant of the conditional measures along the dynamics remains difficult to check outside of the strongly convex case. Another interesting and complementary direction would be to remove the strong log-Sobolev assumptions, in particular those on the conditional measures $\mu^{k+1 \mid [k]}(\cdot \mid x^{[k]})$, by appealing to similar arguments as~\cite{ren2024size}. It would be interesting to consider more general settings, such as the minimizer of general functionals $\mu \mapsto \mc F(\mu)$ considered in~\cite{chen2024uniform, chen2025uniform}; this problem would need significantly new techniques, as the hierarchical argument is not easily adapted to that setting. Finally, the scaling $k^3$ (and the additional polylogarithmic factors) is likely sub-optimal, and reducing these to the standard $k^2$ would be preferred. A similar remark holds for the quadratic scaling $q^2$ compared to the expected $q$ observed in the Gaussian example.

\section*{Acknowledgements}

We thank Sinho Chewi, Daniel Lacker and Songbo Wang for helpful discussions. MSZ was supported by NSERC through a CGS-D award.

\bibliography{main}

\end{document}

%% file: math-macros.tex
\input{_macros.tex}

\global\long\def\bw{\textsf{Ball walk}}%
\global\long\def\dw{\textup{\textsf{Dikin walk}}}%
\global\long\def\sw{\textup{\textsf{Speedy walk}}}%

\global\long\def\E{\mathbb{E}}%

\global\long\def\N{\mathbb{N}}%

\global\long\def\R{\mathbb{R}}%

\global\long\def\Z{\mathbb{Z}}%

\global\long\def\veps{\varepsilon}%

\global\long\def\vol{\textrm{vol}}%

\global\long\def\bs#1{\boldsymbol{#1}}%

\global\long\def\eu#1{\EuScript{#1}}%

\global\long\def\mb#1{\mathbf{#1}}%

\global\long\def\mbb#1{\mathbb{#1}}%

\global\long\def\mc#1{\mathcal{#1}}%

\global\long\def\mf#1{\mathfrak{#1}}%

\global\long\def\ms#1{\mathscr{#1}}%

\global\long\def\mss#1{\mathsf{#1}}%

\global\long\def\msf#1{\mathsf{#1}}%

\global\long\def\on#1{\operatorname{#1}}%

\global\long\def\D{\mathrm{d}}%
\global\long\def\grad{\nabla}%
 
\global\long\def\hess{\nabla^{2}}%
 
\global\long\def\lapl{\triangle}%
 
\global\long\def\deriv#1#2{\frac{d#1}{d#2}}%
 
\global\long\def\pderiv#1#2{\frac{\partial#1}{\partial#2}}%
 
\global\long\def\de{\partial}%
\global\long\def\lagrange{\mathcal{L}}%

\global\long\def\Gsn{\mathcal{N}}%
 
\global\long\def\BeP{\textnormal{BeP}}%
 
\global\long\def\Ber{\textnormal{Ber}}%
 
\global\long\def\Bern{\textnormal{Bern}}%
 
\global\long\def\Bet{\textnormal{Beta}}%
 
\global\long\def\Beta{\textnormal{Beta}}%
 
\global\long\def\Bin{\textnormal{Bin}}%
 
\global\long\def\BP{\textnormal{BP}}%
 
\global\long\def\Dir{\textnormal{Dir}}%
 
\global\long\def\DP{\textnormal{DP}}%
 
\global\long\def\Expo{\textnormal{Expo}}%
 
\global\long\def\Gam{\textnormal{Gamma}}%
 
\global\long\def\GEM{\textnormal{GEM}}%
 
\global\long\def\HypGeo{\textnormal{HypGeo}}%
 
\global\long\def\Mult{\textnormal{Mult}}%
 
\global\long\def\NegMult{\textnormal{NegMult}}%
 
\global\long\def\Poi{\textnormal{Poi}}%
 
\global\long\def\Pois{\textnormal{Pois}}%
 
\global\long\def\Unif{\textnormal{Unif}}%

\global\long\def\Abs#1{\left\lvert #1\right\rvert }%
\global\long\def\Par#1{\left(#1\right)}%
\global\long\def\Brack#1{\left[#1\right]}%
\global\long\def\Brace#1{\left\{ #1\right\} }%

\global\long\def\inner#1{\left\langle #1\right\rangle }%
 
\global\long\def\binner#1#2{\left\langle {#1},{#2}\right\rangle }%

\global\long\def\onenorm#1{\norm{#1}_{1}}%
\global\long\def\twonorm#1{\norm{#1}_{2}}%
\global\long\def\infnorm#1{\norm{#1}_{\infty}}%
\global\long\def\fronorm#1{\norm{#1}_{\text{F}}}%
\global\long\def\nucnorm#1{\norm{#1}_{*}}%
\global\long\def\staticnorm#1{\|{#1}\|}%
\global\long\def\statictwonorm#1{\staticnorm{#1}_{2}}%

\global\long\def\dtv#1#2{d_{\textrm{TV}}\Par{#1,#2}}%

\global\long\def\<{}%

\global\long\def\iff{\Leftrightarrow}%
\global\long\def\chooses#1#2{_{#1}C_{#2}}%
 
\global\long\def\defeq{\triangleq}%
\global\long\def\half{\frac{1}{2}}%
 
\global\long\def\nhalf{\nicefrac{1}{2}}%
 
\global\long\def\textint{{\textstyle \int}}%
\global\long\def\texthalf{{\textstyle \frac{1}{2}}}%
 
\global\long\def\textfrac#1#2{{\textstyle \frac{#1}{#2}}}%


\global\long\def\acts{\circlearrowright}%
 
\global\long\def\bun#1#2{e_{#1} \otimes{} e_{#2}}%
 
\global\long\def\closeopen#1{\lbrack#1 \rparen}%
 
\global\long\def\comp{\mss c}%
 
\global\long\def\deq{\coloneqq}%
 
\global\long\def\e{\mathrm{e}}%
 
\global\long\def\eqas{\overset{\text{a.s.}}{=}}%
 
\global\long\def\eqdef{\mathrel{\overset{\makebox[0pt]{\mbox{\normalfont\tiny def.}}}{=}}}%
 
\global\long\def\gcoeff{\genfrac{\lbrack}{\rbrack}{0pt}{}}%
 
\global\long\def\im{\msf i}%
 
\global\long\def\indep{\mathrel{\text{\scalebox{1.07}{\ensuremath{\perp\mkern-10mu \perp}}}}}%
 
\global\long\def\lr{\leftrightarrow}%
 
\global\long\def\lrarrow{\leftrightarrow}%
 
\global\long\def\mmid{\mathbin{\|}}%
 
\global\long\def\openclose#1{\lparen#1 \rbrack}%
 
\global\long\def\relmid{\mathrel{}\middle|\mathrel{}}%
 
\global\long\def\rest{\upharpoonright}%
 
\global\long\def\simiid{\overset{\text{i.i.d.}}{\sim}}%
 
\global\long\def\toae#1{\xrightarrow[#1]{\text{a.e.}}}%
 
\global\long\def\toas#1{\xrightarrow[#1]{\text{a.s.}}}%
 
\global\long\def\toprob#1{\xrightarrow[#1]{\mathbb{P}}}%
 
\global\long\def\wmaj{\mathbin{\prec_{\msf w}}}%
 
\global\long\def\T{\mathsf{T}}%

\global\long\def\dee{\mathop{\mathrm{d}\!}}%
 
\global\long\def\dt{\,\dee t}%
 
\global\long\def\ds{\,\dee s}%
 
\global\long\def\dx{\,\dee x}%
 
\global\long\def\dy{\,\dee y}%
 
\global\long\def\dz{\,\dee z}%
 
\global\long\def\dv{\,\dee v}%
 
\global\long\def\dw{\,\dee w}%
 
\global\long\def\dr{\,\dee r}%
 
\global\long\def\dB{\,\dee B}%
\global\long\def\dW{\,\dee W}%
\global\long\def\dmu{\,\dee\mu}%
 
\global\long\def\dnu{\,\dee\nu}%
 
\global\long\def\domega{\,\dee\omega}%

\global\long\def\smiddle{\mathrel{}|\mathrel{}}%
\global\long\def\qtext#1{\quad\text{#1}\quad}%
\global\long\def\psdle{\preccurlyeq}%
 
\global\long\def\psdge{\succcurlyeq}%
 
\global\long\def\psdlt{\prec}%
 
\global\long\def\psdgt{\succ}%

\global\long\def\boldone{\mbf{1}}%
\global\long\def\ident{\mbf{I}}%

\global\long\def\eqdist{\stackrel{d}{=}}%
 
\global\long\def\todist{\stackrel{d}{\to}}%
 
\global\long\def\eqd{\stackrel{d}{=}}%
 
\global\long\def\independenT#1#2{\mathrel{\rlap{$#1#2$}\mkern4mu {#1#2}}}%

\global\long\def\ind{\mathds{1}}%

%% file: _macros.tex
\def\balign#1\ealign{\begin{align}#1\end{align}}
\def\baligns#1\ealigns{\begin{align*}#1\end{align*}}
\def\balignat#1\ealign{\begin{alignat}#1\end{alignat}}
\def\balignats#1\ealigns{\begin{alignat*}#1\end{alignat*}}
\def\bitemize#1\eitemize{\begin{itemize}#1\end{itemize}}
\def\benumerate#1\eenumerate{\begin{enumerate}#1\end{enumerate}}

\newenvironment{talign*}
 {\let\displaystyle\textstyle\csname align*\endcsname}
 {\endalign}
\newenvironment{talign}
 {\let\displaystyle\textstyle\csname align\endcsname}
 {\endalign}

\def\balignst#1\ealignst{\begin{talign*}#1\end{talign*}}
\def\balignt#1\ealignt{\begin{talign}#1\end{talign}}

\let\originalleft\left
\let\originalright\right
\renewcommand{\left}{\mathopen{}\mathclose\bgroup\originalleft}
\renewcommand{\right}{\aftergroup\egroup\originalright}

\def\Gronwall{Gr\"onwall\xspace}
\def\Holder{H\"older\xspace}
\def\Ito{It\^o\xspace}
\def\Nystrom{Nystr\"om\xspace}
\def\Schatten{Sch\"atten\xspace}
\def\Matern{Mat\'ern\xspace}

\def\tinycitep*#1{{\tiny\citep*{#1}}}
\def\tinycitealt*#1{{\tiny\citealt*{#1}}}
\def\tinycite*#1{{\tiny\cite*{#1}}}
\def\smallcitep*#1{{\scriptsize\citep*{#1}}}
\def\smallcitealt*#1{{\scriptsize\citealt*{#1}}}
\def\smallcite*#1{{\scriptsize\cite*{#1}}}

\def\blue#1{\textcolor{blue}{{#1}}}
\def\green#1{\textcolor{green}{{#1}}}
\def\orange#1{\textcolor{orange}{{#1}}}
\def\purple#1{\textcolor{purple}{{#1}}}
\def\red#1{\textcolor{red}{{#1}}}
\def\teal#1{\textcolor{teal}{{#1}}}

\def\mbi#1{\boldsymbol{#1}} 
\def\mbf#1{\mathbf{#1}}
\def\mrm#1{\mathrm{#1}}
\def\tbf#1{\textbf{#1}}
\def\tsc#1{\textsc{#1}}

\def\mbiA{\mbi{A}}
\def\mbiB{\mbi{B}}
\def\mbiC{\mbi{C}}
\def\mbiDelta{\mbi{\Delta}}
\def\mbif{\mbi{f}}
\def\mbiF{\mbi{F}}
\def\mbih{\mbi{g}}
\def\mbiG{\mbi{G}}
\def\mbih{\mbi{h}}
\def\mbiH{\mbi{H}}
\def\mbiI{\mbi{I}}
\def\mbim{\mbi{m}}
\def\mbiP{\mbi{P}}
\def\mbiQ{\mbi{Q}}
\def\mbiR{\mbi{R}}
\def\mbiv{\mbi{v}}
\def\mbiV{\mbi{V}}
\def\mbiW{\mbi{W}}
\def\mbiX{\mbi{X}}
\def\mbiY{\mbi{Y}}
\def\mbiZ{\mbi{Z}}

\def\textsum{{\textstyle\sum}} 
\def\textprod{{\textstyle\prod}} 
\def\textbigcap{{\textstyle\bigcap}} 
\def\textbigcup{{\textstyle\bigcup}} 

\def\reals{\mathbb{R}} 
\def\integers{\mathbb{Z}} 
\def\rationals{\mathbb{Q}} 
\def\naturals{\mathbb{N}} 
\def\complex{\mathbb{C}} 

\def\what#1{\widehat{#1}}

\def\twovec#1#2{\left[\begin{array}{c}{#1} \\ {#2}\end{array}\right]}
\def\threevec#1#2#3{\left[\begin{array}{c}{#1} \\ {#2} \\ {#3} \end{array}\right]}
\def\nvec#1#2#3{\left[\begin{array}{c}{#1} \\ {#2} \\ \vdots \\ {#3}\end{array}\right]} 

\def\maxeig#1{\lambda_{\mathrm{max}}\left({#1}\right)}
\def\mineig#1{\lambda_{\mathrm{min}}\left({#1}\right)}

\def\Re{\operatorname{Re}} 
\def\indic#1{\mbb{I}\left[{#1}\right]} 
\def\logarg#1{\log\left({#1}\right)} 
\def\polylog{\operatorname{polylog}}
\def\maxarg#1{\max\left({#1}\right)} 
\def\minarg#1{\min\left({#1}\right)} 
\def\Earg#1{\E\left[{#1}\right]}
\def\Esub#1{\E_{#1}}
\def\Esubarg#1#2{\E_{#1}\left[{#2}\right]}
\def\bigO#1{\mathcal{O}\left(#1\right)} 
\def\littleO#1{o(#1)} 
\def\P{\mbb{P}} 
\def\Parg#1{\P\left({#1}\right)}
\def\Psubarg#1#2{\P_{#1}\left[{#2}\right]}
\def\Trarg#1{\Tr\left[{#1}\right]} 
\def\trarg#1{\tr\left[{#1}\right]} 
\def\Var{\mrm{Var}} 
\def\Vararg#1{\Var\left[{#1}\right]}
\def\Varsubarg#1#2{\Var_{#1}\left[{#2}\right]}
\def\Cov{\mrm{Cov}} 
\def\Covarg#1{\Cov\left[{#1}\right]}
\def\Covsubarg#1#2{\Cov_{#1}\left[{#2}\right]}
\def\Corr{\mrm{Corr}} 
\def\Corrarg#1{\Corr\left[{#1}\right]}
\def\Corrsubarg#1#2{\Corr_{#1}\left[{#2}\right]}
\newcommand{\info}[3][{}]{\mathbb{I}_{#1}\left({#2};{#3}\right)} 
\newcommand{\staticexp}[1]{\operatorname{exp}(#1)} 
\newcommand{\loglihood}[0]{\mathcal{L}} 


\providecommand{\arccos}{\mathop\mathrm{arccos}}
\providecommand{\dom}{\mathop\mathrm{dom}}
\providecommand{\diag}{\mathop\mathrm{diag}}
\providecommand{\tr}{\mathop\mathrm{tr}}
\providecommand{\card}{\mathop\mathrm{card}}
\providecommand{\sign}{\mathop\mathrm{sign}}
\providecommand{\conv}{\mathop\mathrm{conv}} 
\def\rank#1{\mathrm{rank}({#1})}
\def\supp#1{\mathrm{supp}({#1})}

\providecommand{\minimize}{\mathop\mathrm{minimize}}
\providecommand{\maximize}{\mathop\mathrm{maximize}}
\providecommand{\subjectto}{\mathop\mathrm{subject\;to}}

\def\openright#1#2{\left[{#1}, {#2}\right)}

\ifdefined\nonewproofenvironments\else
\ifdefined\ispres\else
\newtheorem{theorem}{Theorem}
\newtheorem{lemma}[theorem]{Lemma}
\newtheorem{corollary}[theorem]{Corollary}
\newtheorem{definition}[theorem]{Definition}
\newtheorem{fact}[theorem]{Fact}
\renewenvironment{proof}{\noindent\textbf{Proof.}\hspace*{.3em}}{\qed \vspace{.1in}}
\newenvironment{proof-sketch}{\noindent\textbf{Proof Sketch}
  \hspace*{1em}}{\qed\bigskip\\}
\newenvironment{proof-idea}{\noindent\textbf{Proof Idea}
  \hspace*{1em}}{\qed\bigskip\\}
\newenvironment{proof-of-lemma}[1][{}]{\noindent\textbf{Proof of Lemma {#1}}
  \hspace*{1em}}{\qed\\}
  \newenvironment{proof-of-proposition}[1][{}]{\noindent\textbf{Proof of Proposition {#1}}
  \hspace*{1em}}{\qed\\}
\newenvironment{proof-of-theorem}[1][{}]{\noindent\textbf{Proof of Theorem {#1}}
  \hspace*{1em}}{\qed\\}
\newenvironment{proof-attempt}{\noindent\textbf{Proof Attempt}
  \hspace*{1em}}{\qed\bigskip\\}
\newenvironment{proofof}[1]{\noindent\textbf{Proof of {#1}}
  \hspace*{1em}}{\qed\bigskip\\}
 
\newtheorem*{remark*}{Remark}
\newenvironment{remark}{\noindent\textbf{Remark.}
  \hspace*{0em}}{\smallskip}
\newenvironment{remarks}{\noindent\textbf{Remarks}
  \hspace*{1em}}{\smallskip}
\fi
\newtheorem{observation}[theorem]{Observation}
\newtheorem{proposition}[theorem]{Proposition}
\newtheorem{claim}[theorem]{Claim}
\newtheorem{assumption}{Assumption}
\theoremstyle{definition}
\newtheorem{example}[theorem]{Example}
\theoremstyle{remark}
\newtheorem{intuition}[theorem]{Intuition}
\fi
\makeatletter
\@addtoreset{equation}{section}
\makeatother
\def\theequation{\thesection.\arabic{equation}}

\newcommand{\cmark}{\ding{51}}

\newcommand{\xmark}{\ding{55}}

\newcommand{\eq}[1]{\begin{align}#1\end{align}}
\newcommand{\eqn}[1]{\begin{align*}#1\end{align*}}
\renewcommand{\Pr}[1]{\mathbb{P}\left( #1 \right)}
\newcommand{\Ex}[1]{\mathbb{E}\left[#1\right]}

\newcommand{\matt}[1]{{\textcolor{blue}{[Matt: #1]}}}